\documentclass[12pt]{amsart}
\usepackage{graphicx}
\usepackage{epstopdf}
\usepackage[mathscr]{eucal}
\usepackage{amssymb}
\usepackage{slashbox}

\usepackage{amsthm}

\theoremstyle{plain}
\newtheorem{prop}{Proposition}[section]
\newtheorem{coro}[prop]{Corollary}

\newtheorem{lemm}[prop]{Lemma}

\newtheorem{thm}[prop]{Theorem}
\theoremstyle{remark}
\newtheorem{ex}[prop]{Example}
\theoremstyle{remark}
\newtheorem{defn}[prop]{Definition}

\newtheorem{rem}[prop]{Remark}

\DeclareMathOperator{\sign}{sign}

\def\mcg#1;#2{\Gamma_{#1,#2}}
\def\fg#1;#2{\Pi_{#1,#2}}
\def\tb#1;#2{\mathscr{K}_{\frac{#1}{#2}}}

\begin{document}

\title[Twisting Some Classes of Links]
{Twisting Some Classes of Links}

\keywords{Adequate links, homogeneous links, alternative links, twisting}

\author{Khaled Qazaqzeh}
\address{Current Address: Department of Mathematics, Faculty of Science, Yarmouk University, Irbid, Jordan, 21163}
\email{qazaqzeh@yu.edu.jo}
\urladdr{http://faculty.yu.edu.jo/qazaqzeh}
\address{Address: Department of Mathematics, Faculty of Science,  Kuwait
University, P. O. Box 5969 Safat-13060, Kuwait, State of Kuwait}
\email{khaled.qazaqzeh@ku.edu.kw}

\author{Ahmad Al-Rhayyel}
\address{Address: Department of Mathematics, Faculty of Science, Yarmouk University, Irbid, Jordan, 21163}
\email{al-rhayyel@yu.edu.jo}
\urladdr{https://faculty.yu.edu.jo/Alrhayyel}

\date{22/11/2022}
\subjclass{57K10, 57K14}
\begin{abstract}
We apply the twisting technique that was first introduced in \cite{CK} and later generalized in \cite{QCQ} to obtain an infinite family of adequate, homogeneous  or alternative links from a given adequate, homogeneous or alternative link, respectively. Thus we conclude that these three classes of links  are infinite.
\end{abstract}

\maketitle

\section{introduction}

Many classes of links have been defined in trying to generalize the class of alternating links.
In particular, Lickorish and Thistlewaite in \cite{LT} introduced the class of adequate
links and Cromwell in \cite{C} introduced the class of homogeneous links. Later, Kauffman in
\cite{K2} introduced the class of alternative links. These
classes of links are natural generalization of alternating links each in its own aspect.
We recall the definition of these classes for the sake of making this paper more self-contained.

First, we recall the definition of adequate links as in \cite{LT}. 
Let $L$ be a link of a diagram $D$ with crossings
$c_{1},c_{2},\ldots,c_{n}$. A state of the diagram $D$ is a
function $s: \{c_{1},c_{2},\ldots,c_{n}\}\rightarrow \{1,-1\}$, and the
state diagram $sD$ is the diagram obtained from the diagram $D$
after smoothing all of its crossings according to the state $s$. We apply the $A$- or
$B$-smoothing to the crossing $c_{i}$ if $s(c_{i}) = 1$ or $s(c_{i})
= -1$, respectively according to the scheme in Figure \ref{figure1}. Note that
$sD$ consists of simple closed curves and the number of components
will be denoted by $|sD|$.

There are two special state diagrams $s_{A}D$ and $s_{B}D$. The
first one is obtained when $s_{A}(c_{i}) = 1$  and the
second one is obtained when $s_{B}(c_{i}) = -1$ for all $i$ such that $1\leq i \leq n$. The
link diagram is $A$-adequate if $|s_{A}D|> |sD|$ for every state
diagram $sD$ with $\sum_{i=1}^{n}s(c_{i}) = n-2$  and is $B$-adequate if $|s_{B}D|> |sD|$ for
every state diagram $sD$ with $\sum_{i=1}^{n}s(c_{i}) = 2-n$. This is
equivalent to say that the two arcs after smoothing each crossing in $s_{A}D$ or $s_{B}D$
belong to two different components, respectively. The link diagram is adequate
if it is $A$- and $B$-adequate at the same time and the
link is adequate if it admits a diagram that is adequate. Figure \ref{figure1} below shows a link diagram at the crossing $c$ with $\sign(c) =
1$ and both of its smoothings, $A$- and $B$-smoothings,
respectively.

\begin{figure} [h]
\begin{center}
\includegraphics[width=8cm,height=1.5cm]{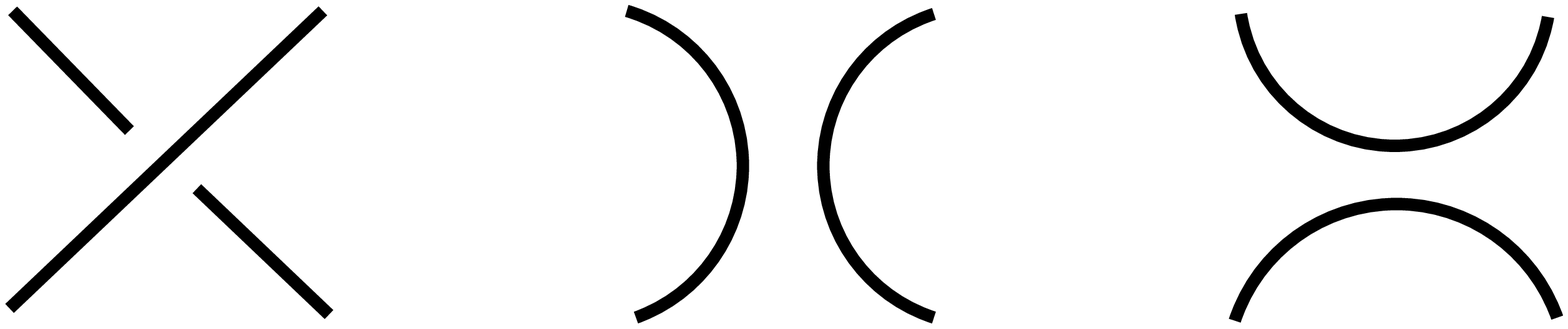}
\end{center}
\caption{}\label{figure1}
\end{figure}

Thistlewaite in \cite{T1} proved that any adequate diagram of an adequate link has the minimal
crossing number among all diagrams that represent the same link using the Kauffman polynomial. The same result is
obtained using the Jones polynomial \cite{LT}. Moreover, the result
of \cite{T1} implies that the minimal crossing diagram of an adequate link is also adequate.

Second, we recall the definition of homogeneous links as in \cite{C}. Seifert in \cite{Se} defined an algorithm for constructing an
orientable surface spanning the link $L$ from its oriented diagram
$D$. The deformation retraction of this surface is a graph $G$. The
vertices of $G$ correspond to the Seifert circles after collapsing
them to points. The edges of $G$ correspond to the twisted
rectangles that connect the Seifert circles, and hence the crossings
in $D$. The edges of $G$ can be equipped with signs according to the type of the crossing
given in Figure \ref{figure2}. The resulting graph with signs
is called Seifert graph.

\begin{figure} [h]
\begin{center}
\includegraphics[scale=0.10]{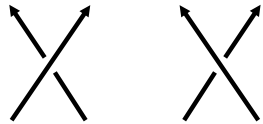}\hspace{1cm}
\reflectbox{\includegraphics[scale=0.10]{figure2.eps}}
\end{center} \caption{Type I and II crossings,
respectively.}\label{figure2}
\end{figure}

In any graph, a cut vertex $v$ is a vertex whose removal increases the number of components. If the components of $G-v$ are
$G_{1},G_{2},\ldots,G_{n}$. Then $G_{1}\cup v,G_{2}\cup v, \ldots,
G_{n}\cup v$ are subgraphs of $G$ obtained by cutting $G$ at $v$.
Cutting $G$ at each of its cut vertices yields a family of subgraphs,
no one of which has a cut vertex. These subgraphs are called the blocks of the graph 
$G$.

A Seifert graph is homogeneous if all of the edges in each block have
the same sign. A link diagram is homogeneous if the corresponding
Seifert graph is homogeneous. The link is homogeneous if it admits a
diagram that is homogeneous.

Finally, we recall the definition of alternative links via the characterization introduced in
\cite[Theorem\,19]{Si}. This characterization is given in terms of enhanced
checkerboard graph. This graph $\Phi(D)$ is a signed planar digraph constructed from the link
diagram $D$ in the following way. We color the regions of the diagram into
two colors black and white such that regions which share an arc have different
colors. We then place a vertex in each region colored according to the color of the region.
Two vertices are connected by a directed edge with a sign according to the scheme in Figure \ref{figure5}.

\begin{figure}[htbp]
    \begin{center}
\includegraphics[scale=0.17]{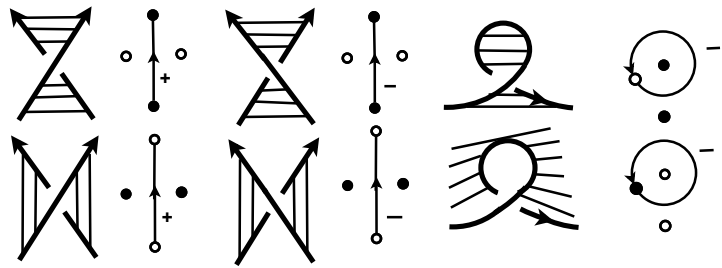}
\end{center}
\caption{The edges and vertices of the graph $\Phi(D)$}\label{figure5}
\end{figure}

A planar signed digraph $G$ is said to be alternative if it does not contain any walk with
positive and negative edges. The link diagram is alternative if the corresponding graph $\Phi(D)$
is alternative. The link is alternative if it has a diagram that is alternative.

In this paper, we explain how to construct an infinite family of adequate, homogeneous, or alternative links 
from a given adequate, homogeneous or alternative link, respectively using the twisting technique. The twisting technique  introduced for the first time in \cite{CK} to construct an infinite family of quasi-alternating links from a given quasi-alternating link by replacing a particular crossing by some rational tangle. This technique was later generalized in \cite{QCQ}  by replacing the particular crossing by a block of rational tangles not just a rational tangle. 

\section{Rational Tangles and Determinant}

In this section, we briefly recall and review some properties of rational tangles that will be used in the rest of this paper. It is well-known that each rational tangle is characterized by the continued fraction of some rational number that is called the slope of the rational tangle. The continued fraction of the rational number $\frac{\beta}{\alpha}$
is given by
\begin{align*}\label{continued}
\frac{\beta}{\alpha} =
\cfrac{1}{a_{n}+\cfrac{1}{a_{n-1}+\cfrac{1}{\ddots
+\cfrac{1}{a_{1}}}}},
\end{align*}
where $a_1,a_{2}, \dots, a_n$ are called the integer denominators of the continued
fraction. This continued fraction will be denoted by the sequence of
integers $[a_{1},a_{2},\ldots,a_{n}]$. It is a simple fact that each
rational number admits such a finite continued fraction.

A rational tangle is a proper embedding of the disjoint union of two
arcs into a cube, the embedding sends the end points of the two
arcs to four marked points NE, NW, SE, and SW with the symbols referring to
the compass directions on the cube's boundary. 

The characterization of a rational tangle by the continued fraction of its slope is given by the scheme as shown in Figure \ref{figure3}.
The $i$-th box in Figure \ref{figure3} corresponds to $|a_{i}|-$
horizontal crossings of positive sign if $a_{i}$ is positive and $i$ is odd or $a_{i}$
is negative and $i$ is even and otherwise it corresponds to $|a_{i}|-$horizontal
crossings of negative sign, where the sign of the crossing is according to the scheme in Figure \ref{figure1}. 

\begin{figure}[htbp]
    \begin{center}
\includegraphics[scale=0.16]{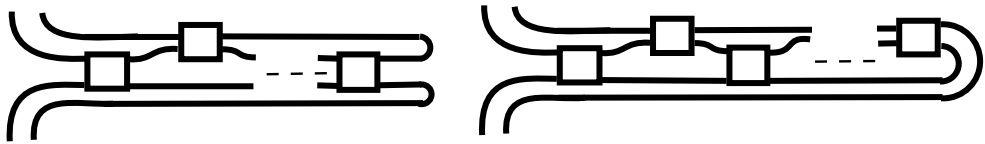}
\end{center}
\caption{The closure of a rational tangle according to the number of denominators being odd or even respectively}\label{figure3}
\end{figure}

The product of rational tangles is defined by placing one tangle on top of the other tangle by identifying the end points of the top tangle marked with SE, SW to the end points of bottom tangle marked with NE, and
NW, respectively.  Also, the sum of two rational tangles is defined by juxtaposing the first
tangle next to the second one and identifying the end points of the
left tangle marked with NW, and SW to the end points of
the right tangle marked with NE, and SE, respectively. A block of
rational tangles is a finite product or/and sum of rational tangles. For further discussion of rational tangles and their operations, we refer the reader to \cite{KaLa}, and \cite{KaLo}. The following lemma lists some of the well-known properties of the continued fraction that can be found in many references for example \cite{KaLa, KaLo, QCQ}.

\begin{lemm}\label{properties}
\begin{enumerate}
\item If $[a_{1},a_{2},\ldots,a_{n}]$ is a continued
fraction of $\frac{\beta}{\alpha}$, then $[-a_{1},-a_{2},\ldots,
-a_{n}]$ is a continued fraction of $\frac{-\beta}{\alpha}$.

\item If $\frac{\beta}{\alpha}$ is a positive rational number, then there
is a continued fraction $[a_{1},a_{2},\ldots,a_{n}]$ of nonnegative
integers.


\item The rational tangle of corresponding continued fraction $[a_{1},a_{2}, \ldots, a_{n-1}, \pm 1]$
is isotopic to the rational tangle of corresponding continued fraction $[a_{1},a_{2}, \ldots, a_{n-1}\pm 1]$.
\end{enumerate}
\end{lemm}


The determinant of a link is an invariant that can be defined in terms of the number of signed spanning tress
of the Tait graph. The Tait graph $G(L)$ associated to the link diagram $L$ is obtained from the
checkerboard coloring of the regions of the link diagram. The vertices of $G(L)$ are the shaded
regions and the signed edges correspond to crossings according to the scheme in Figure \ref{figure4}.

\begin{figure}[htbp]
    \begin{center}
\includegraphics[scale=0.12]{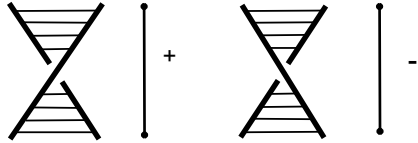}
\end{center}
\caption{The signs of the edges of the Tait graph $G(L)$}\label{figure4}
\end{figure}

For a given link $L$ at the crossing $c$, we let $L_{0}$ and $L_{1}$ be the links obtained by applying the $A-$ and $B-$smoothings respectively to the link $L$ at the given crossing
according to the scheme in Figure \ref{figure1}.
We can choose checkerboard coloring of the regions of $L$ so that the edge $e$
that corresponds to $c$ is positive in $G(L)$. It is well-known that
the spanning trees of $G(L)$ are in one-to-one correspondence of the
spanning trees of $G(L_{0})$ and $G(L_{1})$. In particular, any
spanning tree of $G(L)$ corresponds to a spanning tree of $G(L_{0})$
if it does not contain $e$ and to a spanning tree of $G(L_{1})$ in the other
case.

The number of positive edges in any spanning tree $T$ of $G$ will be
denoted by $v(T)$ and the number of spanning trees of $G$ with $v(T)
= v$ will be denoted by $s_{v}(L)$. The numbers $\sum\limits_{v}
(-1)^{v} s_{v-1}(L_{0})$, $\sum\limits_{v} (-1)^{v} s_{v}(L_{1})$
will be denoted by $x$ and $y$ respectively for the following lemma whose proof can be obtained by adopting the proof of \cite[Theorem\,2.1]{CK} with the appropriate modifications. This lemma may considered as a natural generalization of \cite[Lemma\,3.1]{QCQ} and its proof.

\begin{lemm}\label{help}
Let $L^{*}$ be the link obtained from $L$ by replacing the crossing $c$ by a
rational tangle of slope $\frac{\beta}{\alpha}$ with corresponding continued fraction of
denominators of the same sign as the sign of $c$.
Then we have
\[ \det(L^{*}) =
  \begin{cases}
  \left|\alpha \det(L_{0}) +\sign(xy) \beta \det(L_{1})\right|,
     & \text{if $\sign(c) = 1 $},
     \\
    {\left|\alpha \det(L_{1}) +\sign(xy) \beta \det(L_{0})\right|}, & \text{if $\sign(c) = -1
    $}.
  \end{cases}
\]
\end{lemm}

\section{Main Results}

We recall the twisting technique that was first
introduced in \cite[Page\,2452]{CK} and later generalized in \cite{QCQ}. This technique replaces a particular crossing in a link diagram by some rational tangle or sum or product of rational tangles in the class of quasi-alternating links.
\begin{defn}
The rational tangle associated to the continued fraction
$[a_{1},a_{2},\ldots,a_{n}]$ of denominators of the same sign extends the crossing $c$ if the sign of the crossing $c$ is the same as the common sign of the denominators of this continued fraction, 
where the sign of $c$ is defined according to the scheme in Figure \ref{figure1}. A block of rational tangles extends the crossing $c$ if each rational tangle in the block extends it.
\end{defn}

The above definition can be modified to the case of oriented rational tangles as follows:

\begin{defn}
The orientation of a rational tangle associated to the continued fraction
$[a_{1},a_{2},$ $\ldots,a_{n}]$ of denominators of the same sign extends the orientation of the
crossing $c$ if the crossings in the $i$-th box are of the same type as the type of the oriented crossing $c$ according to the scheme in Figure \ref{figure2} if $i$ is odd and of the opposite type if $i$ is even.  
An oriented rational tangle extends the oriented crossing $c$ if it extends it with no
orientation and its orientation extends the orientation of $c$.
The block of oriented rational tangles extends the oriented crossing $c$ if each oriented
rational tangle in the block extends it.
\end{defn}


We point out that not all rational tangles that extend the crossing $c$ can be equipped with
an orientation that extends the orientation of $c$. Now we state the first main theorem in this paper:

\begin{thm}\label{new1}
Let $L$ be an adequate link with adequate diagram $D$. If
$L^{*}$ is the link obtained by replacing any crossing in $D$ by a block of
rational tangles that extends it, then $L^{*}$ is adequate with corresponding adequate diagram
$D^{*}$.
\end{thm}

\begin{proof}
Let $c$ be any crossing of the adequate diagram $D$.  By taking the mirror image if necessary, we may assume that $\sign(c) = 1$ and later use the fact that the diagram is adequate iff its mirror image is adequate to obtain the required result.

We prove the result first in the case when the crossing is
replaced by an integer tangle that extends it. Let $D^{m}$ be the link
diagram obtained from $D$ after replacing $c$ by an integer tangle of $m$ vertical or horizontal crossings that extends this crossing. In the special states $s_{A}D$ and $s_{B}D$, the two arcs in smoothing any crossing
belong to two different components since $D$ is adequate. In particular, this is
the case for the two arcs in smoothing $c$. Now it is obvious that
this also holds for any crossing of $D^{m}$ to obtain $s_{A}D^{m}$ or $s_{B}D^{m}$. 

Now we prove that any crossing of this integer tangle
can be replaced by a rational tangle that extends it and at the same time extends the 
crossing $c$. For this end, let $[a_{1},a_{2},\ldots, a_{n}]$ be the continued fraction of
some rational tangle that extends $c$. We use induction on the number of denominators of the
continued fraction to prove our claim. The above argument implies the result for $n=1$.
Now the rational tangle corresponding to the continued fraction $[a_{1},a_{2},\ldots, a_{n-1}+1]$
is isotopic to the rational tangle corresponding  to the continued fraction
$[a_{1},a_{2},\ldots, a_{n-1}, 1]$ from the fourth part of Lemma \ref{properties}. Now, it is not too hard to see 
that the diagram obtained by replacing the crossing $c$ by the rational tangle of continued fraction $[a_{1},a_{2},\ldots, a_{n-1}+1]$ is adequate iff the diagram obtained by replacing the crossing $c$ by the rational tangle of continued fraction $[a_{1},a_{2},\ldots, a_{n-1},1]$ is adequate. Finally, the result follows using the induction hypothesis and the result for the case when $n=1$.
\end{proof}
\begin{coro}
For the links $L$ and $L^{*}$, we have $c(L^{*}) = c(L) + c(b) = c(D) +c(b)$ where $c(b)$ is the number of crossings of the block of rational tangles that extends some crossing $c$ in the link diagram $D$.
\end{coro}
\begin{proof}
The result follows as a consequence of the fact that the adequate  reduced diagram of a link has minimal crossing number in \cite{T1}.
\end{proof}
\begin{coro}
A given adequate link yields an infinite family of distinct adequate
links.
\end{coro}
\begin{proof}
The set of links $\{L^{*}\}$ is infinite as it contains an infinite subset of links of distinct crossing number.  
\end{proof}

Now if we let $\mathfrak{M}\langle D\rangle$ and $\mathfrak{m}\langle D\rangle$
to denote the maximum and minimum powers of $A$ that occur in the Kauffman bracket of the adequate link diagram $D$, then in the following proposition we can give the maximum and minimum powers of $A$ that occur in the Kauffman bracket of the link diagram $D^{*}$ in terms of the corresponding ones of $D$.

\begin{prop}
Let $D^{*}$ be the adequate diagram obtained by replacing a crossing of positive sign in the adequate diagram $D$
by a block of rational tangles that extends it consisting of a product of $l$ tangles each of
which is a sum of $k_{n}$ rational tangles for $1 \leq n \leq l$ with corresponding continued fraction
$[a_{1}^{ij},a_{2}^{ij},\ldots, a_{m_{ij}}^{ij}]$ for $1 \leq i \leq l$ and $1 \leq j \leq k_{i}$, then we have
\[ \mathfrak{M}\langle D^{*}\rangle = \mathfrak{M}\langle D\rangle +
\sum\limits_{i=1}^{l}\sum\limits_{j=1}^{k_{i}}\sum\limits_{r=1}^{m_{ij}} a_{r}^{ij}
+ 2 \sum\limits_{i=1}^{l}\sum\limits_{j=1}^{k_{i}} T_{+}^{ij} + 2 \sum\limits_{n=1}^{l} (k_{n}-1)-1,\]
and
\[
\mathfrak{m}\langle D^{*}\rangle = \mathfrak{m}\langle D\rangle -
\sum\limits_{i=1}^{l}\sum\limits_{j=1}^{k_{i}}\sum\limits_{r=1}^{m_{ij}} a_{r}^{ij}
- 2 \sum\limits_{i=1}^{l}\sum\limits_{j=1}^{k_{i}} T_{-}^{ij} -2l +3,
\]
where $T_{+}^{ij} = \sum\limits_{r=1}^{s}a_{2r-1}^{ij} -\frac{1}{2}((-1)^{m_{ij}-1}+1)$ and
$T_{-}^{ij} = \sum\limits_{r=1}^{s}a_{2r}^{ij} -\frac{1}{2}((-1)^{m_{ij}}+1)$ with $2s-1\leq m_{ij}$ and $2s\leq m_{ij}$. In the case that
the block consists of a sum of $l$ tangles each of which consisting of a product of $k_{n}$ rational tangles
$1 \leq n \leq l$, with corresponding continued fraction
$[a_{1}^{ij},a_{2}^{ij},\ldots, a_{m_{ij}}^{ij}]$ for $1 \leq i \leq l$ and $1 \leq j \leq k_{i}$, then we just interchange the terms $2 \sum\limits_{n=1}^{l} (k_{n}-1)-1$ and $-2l +3$ after changing their signs. 

\end{prop}

\begin{proof}
It is not to hard to prove the following facts
$s_{A}D^{*} = s_{A}D + \sum\limits_{i=1}^{l}\sum\limits_{j=1}^{k_{l}} T_{+}^{ij} + \sum\limits_{n=1}^{l}(k_{n} - 1),$
and $s_{B}D^{*} = s_{B}D + \sum\limits_{i=1}^{l}\sum\limits_{j=1}^{k_{l}} T_{-}^{ij}+ (l-1)$ in the first case and $s_{A}D^{*} = s_{A}D + \sum\limits_{i=1}^{l}\sum\limits_{j=1}^{k_{l}} T_{+}^{ij} + (l-1),$
and $s_{B}D^{*} = s_{B}D + \sum\limits_{i=1}^{l}\sum\limits_{j=1}^{k_{l}} T_{-}^{ij}+ \sum\limits_{n=1}^{l}(k_{n}-1)$ in the second case. Now the result follows using the formulas for the maximum and minimum powers of $A$ that occur in the Kauffman bracket of any adequate link diagram $D$ given in \cite[Proposition\,1]{LT}.
\end{proof}

\begin{rem}
A similar result  can be obtained in the case that the crossing is of negative sign if we work with the mirror image
of the link diagram. Note that the extreme powers of the Kauffman bracket of the link diagram are the
opposite of the extreme powers of the Kauffman bracket of the mirror image of the link diagram and
vice versa.
\end{rem}
\begin{ex}
It is easy to see that the pretzel link diagram
$D(\underset{r-times}{\underbrace{2,2,\ldots,2}};\underset{s-times}{\underbrace{-2,-2,\ldots,-2}})$
with $r, s\geq 2$ is adequate (see the only example in \cite[Page\,529]{LT}). Here 2, and -2 denote the number of
positive and negative crossings, respectively. Therefore, the Montesinos link diagram $M((\alpha_{1}, \beta_{1}),
(\alpha_{2}, \beta_{2}), \ldots ,(\alpha_{r}, \beta_{r});(\gamma_{1}, \delta_{1}),
(\gamma_{2}, \delta_{2}), \ldots ,(\gamma_{s}, \delta_{s}))$ with $r, s\geq 2$ is adequate according
to the above result. In this notation,
$\alpha_{i}, \beta_{i}$ are coprime positive integers with $\alpha_{i}
> 1$ for $1\leq i \leq r$ and $\gamma_{i}, \delta_{i}$ are coprime integers with $\frac{\gamma_{i}}{\delta_{i}}< 0$
and $|\gamma_{i}|> 1$ for $1\leq i\leq s$. This confirms the result of \cite[Section\,4]{LT}.

\end{ex}

\begin{thm}\label{new2}
Let $L$ be a homogeneous link with a homogeneous diagram $D$. If $L^{*}$ is the link obtained from $D$ by replacing any oriented crossing $c$ by a block of oriented rational tangles that extends it, then $L^{*}$ is homogeneous with homogeneous diagram $D^{*}$.
\end{thm}

\begin{proof}

Let $c$ be any oriented crossing of the homogeneous diagram $D$.  By taking the mirror image if necessary, we may assume that $c$ corresponds to a positive edge in the Seifert graph according to the scheme in Figure \ref{figure2} and later use the fact that the diagram is homogeneous iff its mirror image is homogeneous to obtain the required result.

We prove the result first in the case when the oriented crossing is
replaced by an oriented integer tangle that extends $c$. Let $D^{m}$ be the link
diagram obtained from $D$ after replacing the oriented crossing $c$ by an oriented integer tangle of $m$ vertical or horizontal crossings that extends it.
We show that $D^{m}$ is homogeneous if $D$ is homogeneous. According to orientation, the Seifert graph $G_{m}$
corresponding to $D^{m}$ is exactly the Seifert graph $G$
corresponding to $D$ where the edge that corresponds to $c$ is replaced
by $m$ parallel edges or $m$ collinear edges of the same
sign. Therefore, $G_{m}$ is homogeneous if $G$ is homogeneous.

Now we prove that any oriented crossing of this vertical or horizontal integer tangle
can be replaced by an oriented rational tangle that extends it and at the same time extends the
oriented crossing $c$. For this end, let $[a_{1},a_{2},\ldots, a_{n}]$ be the continued fraction of
some oriented rational tangle that extends $c$. We use induction on the number of denominators of the
continued fraction to prove our claim. The above argument implies the result for $n=1$.
Now the oriented rational tangle of corresponding continued fraction $[a_{1},a_{2},\ldots, a_{n-1}+1]$
is isotopic to the oriented rational tangle of corresponding continued fraction $[a_{1},a_{2},\ldots, a_{n-1}, 1]$
from the fourth part of Lemma \ref{properties}. Now it is not too hard to see 
that the diagram obtained by replacing the crossing $c$ by the rational tangle of continued fraction $[a_{1},a_{2},\ldots, a_{n-1}+1]$ is homogeneous iff the diagram obtained by replacing the crossing $c$ by the rational tangle of continued fraction $[a_{1},a_{2},\ldots, a_{n-1},1]$ is homogeneous. Finally, the result follows using the induction hypothesis
and the case for $n=1$.

\end{proof}
\begin{coro}
The diagram $D$ is homogeneous iff $D^{*}$ is homogeneous.
\end{coro}
\begin{proof}
The result follows directly since the Seifert graph of $D$ is homogenous iff the Seifert graph of $D^{*}$ is homogenous.
\end{proof}
\begin{coro}
A given homogeneous link yields an infinite family of distinct
homogeneous links.
\end{coro}
\begin{proof}
The set of links $\{L^{*}\}$ is infinite as it contains an infinite subset of links of distinct determinants as a result of Lemma \ref{help}.
\end{proof}
As a direct consequence of the fact that the class of positive links is a subclass of 
the class of homogeneous links, we obtain the following:
\begin{coro}
Let $L$ be a positive link with a positive 
link diagram $D$. If $L^{*}$ is the link obtained from $D$ by replacing
any oriented crossing $c$ by block of oriented rational tangles that extends
it, then $L^{*}$ is positive link  with positive link diagram
$D^{*}$.
\end{coro}

\begin{thm}\label{new3}
Let $L$ be an alternative link with a alternative diagram $D$. If link $L^{*}$ is the link obtained from $D$ by replacing any oriented crossing $c$ by a block of oriented rational tangles that extends it, then $L^{*}$ is alternative
with alternative diagram $D^{*}$.
\end{thm}

\begin{proof}
Let $c$ be any oriented crossing of the alternative diagram $D$.  By taking the mirror image if necessary, we can assume that $c$ corresponds to a positive edge in the enhanced checkerboard graph according to the scheme in Figure \ref{figure5} and later use the fact that the diagram is alternative iff its mirror image is alternative to obtain the required result.

We prove the result first in the case when the crossing is
replaced by an oriented integer tangle that extends $c$. Let $D^{m}$ be the link
diagram obtained from $D$ after replacing the oriented $c$ by an oriented vertical or horizontal oriented integer tangle of $m$ crossings that extends it.
According to orientation, the enhanced signed graph $\Phi(G_{m})$
corresponding to $D^{m}$ is exactly the enhanced signed graph $\Phi(G)$
corresponding to $D$ where the edge corresponds to $c$ is replaced
by $m$ parallel edges or $m$ collinear edges of the same
sign. Therefore, $\Phi(G_{m})$ is alternative if $\Phi(G)$ is alternative.

Now we prove that any oriented crossing of this vertical or horizontal integer tangle
can be replaced by an oriented rational tangle that extends it and at the same time extends the oriented crossing $c$. For this end, let $[a_{1},a_{2},\ldots, a_{n}]$
be the continued fraction of some oriented rational tangle that extends $c$.
We use induction on the number of denominators of the continued fraction to prove our claim.
The above argument implies the result for $n=1$. Now the oriented rational tangle of corresponding
continued fraction $[a_{1},a_{2},\ldots, a_{n-1}+1]$ is isotopic to the oriented rational tangle
of corresponding continued fraction $[a_{1},a_{2},\ldots, a_{n-1}, 1]$
from the fourth part of Lemma \ref{properties}. Now it is not too hard to see 
that the diagram obtained by replacing the crossing $c$ by the rational tangle of continued fraction $[a_{1},a_{2},\ldots, a_{n-1}+1]$ is alternative iff the diagram obtained by replacing the crossing $c$ by the rational tangle of continued fraction $[a_{1},a_{2},\ldots, a_{n-1},1]$ is alternative. Finally, the result follows using the induction hypothesis
and the case for $n=1$.

\end{proof}
\begin{coro}
The diagram $D$ is alternative iff $D^{*}$ is alternative.
\end{coro}
\begin{proof}
The result follows directly since the enhanced checkerboard graph $\Phi(D)$ is alternative iff the enhanced checkerboard graph $\Phi(D^{*})$.
\end{proof}
\begin{coro}
A given alternative link yields an infinite family of distinct
alternative links.
\end{coro}
\begin{proof}
The set of links $\{L^{*}\}$ is infinite as it contains an infinite subset of links of distinct determinants as a result of Lemma \ref{help}.
\end{proof}

\begin{ex}
The first non-alternating adequate knot in Rolfsen's knot table of
10 crossings or less is the knot $10_{152}$. It is also homogeneous and alternative
of a homogeneous adequate alternative diagram given in Figure \ref{figure7}. 
We use the table in \cite{CL} to identify some of the knots that can be obtained by replacing a crossing in 
the given knot diagram of $10_{152}$ by a block of rational tangles of three crossings that extends it. These knots are $12n437, 12n558, 12n0679$, $12n0680, 12n0688$, and $12n0689$.

\begin{figure}[h]
    \begin{center}
\includegraphics[scale=0.14]{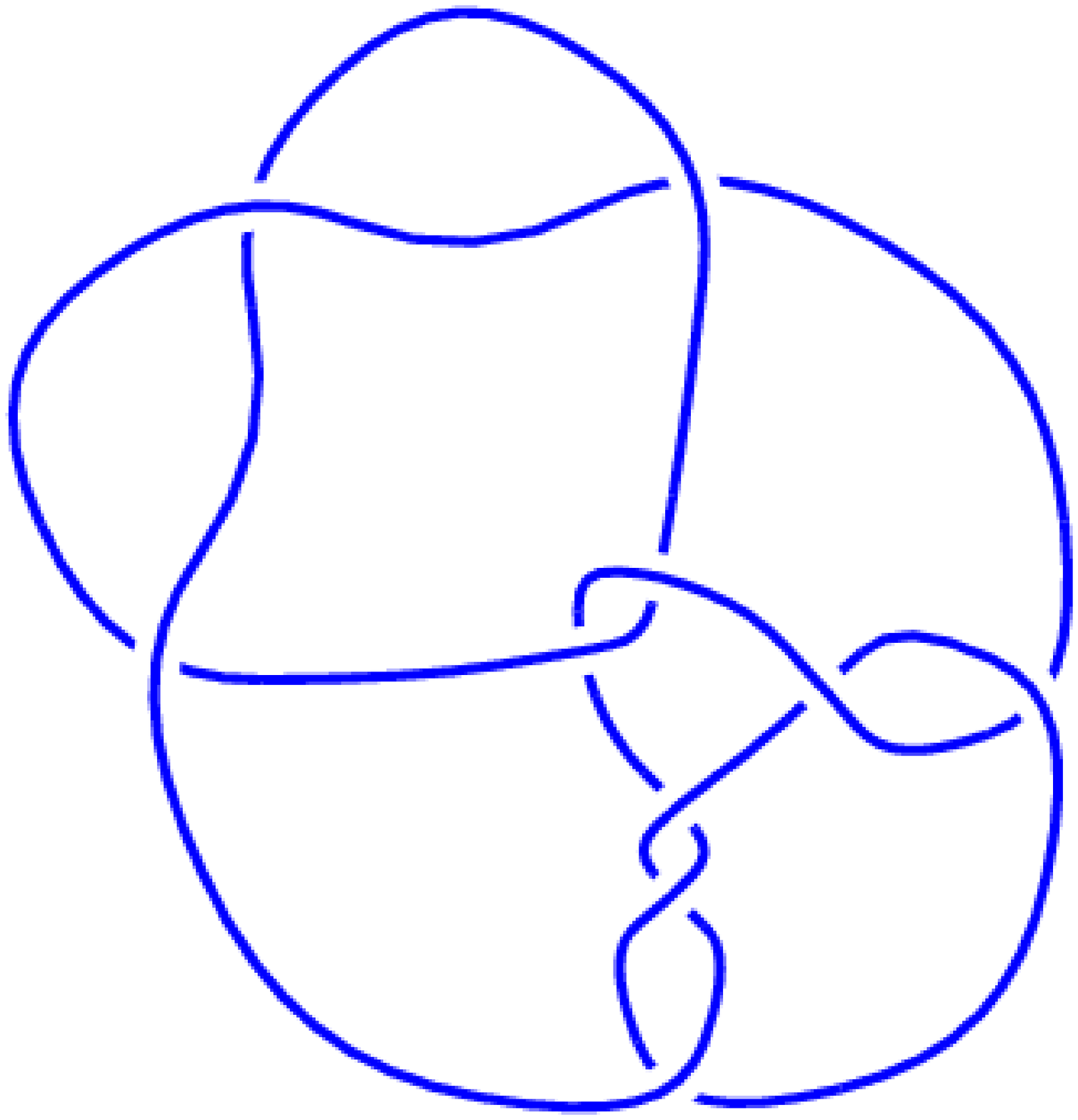}
\includegraphics[scale=0.14]{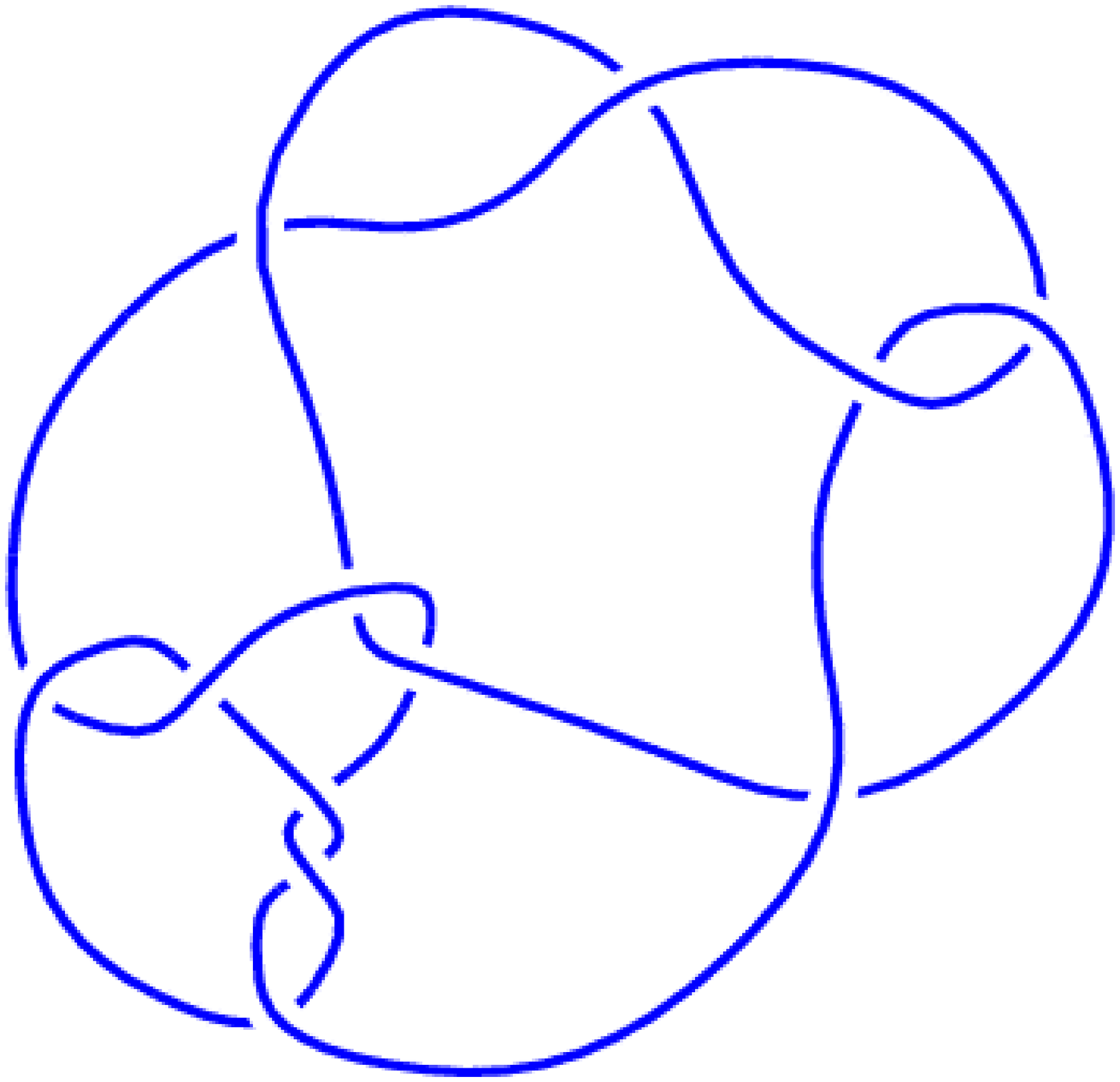}
\includegraphics[scale=0.14]{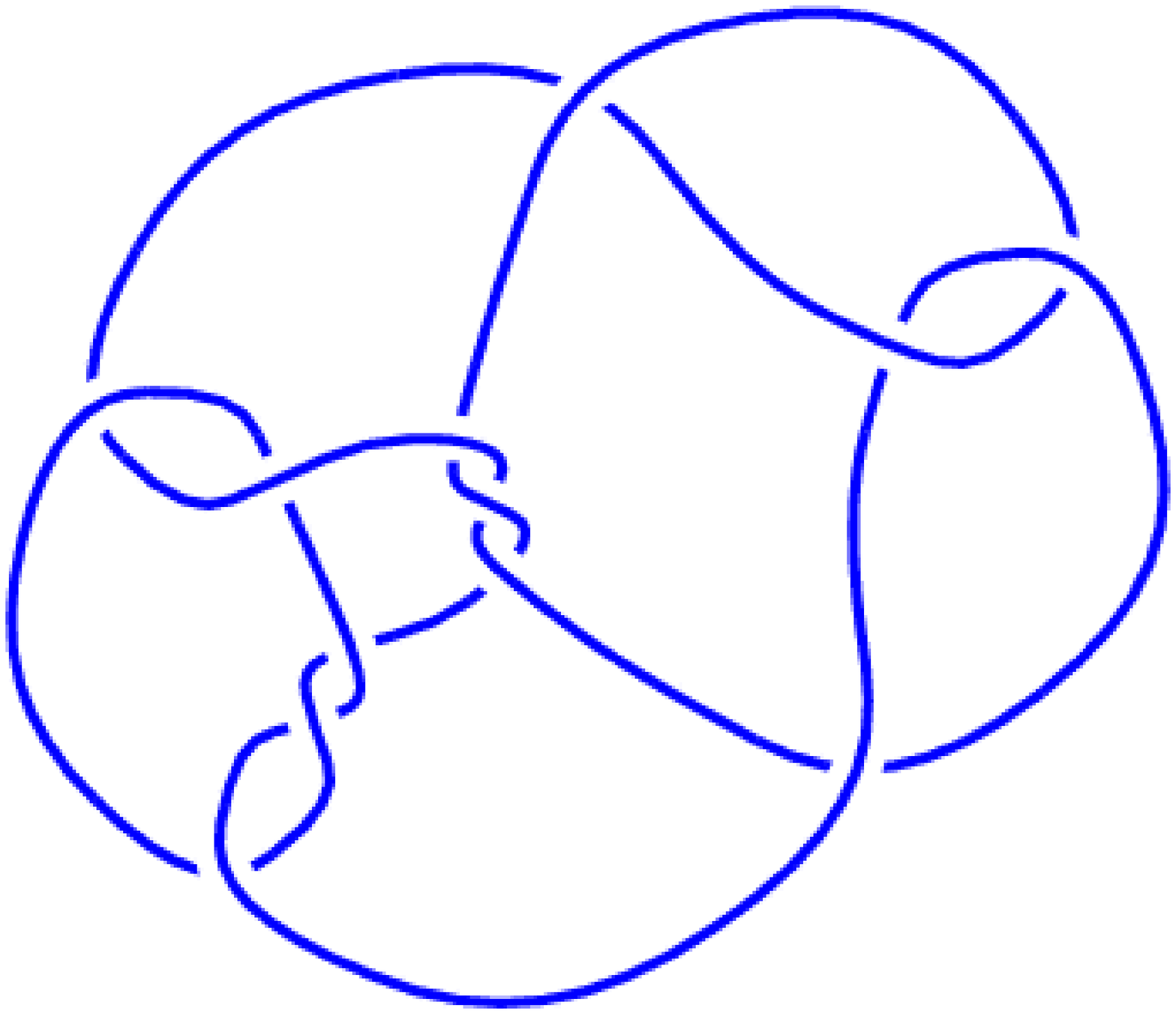}
\includegraphics[scale=0.14]{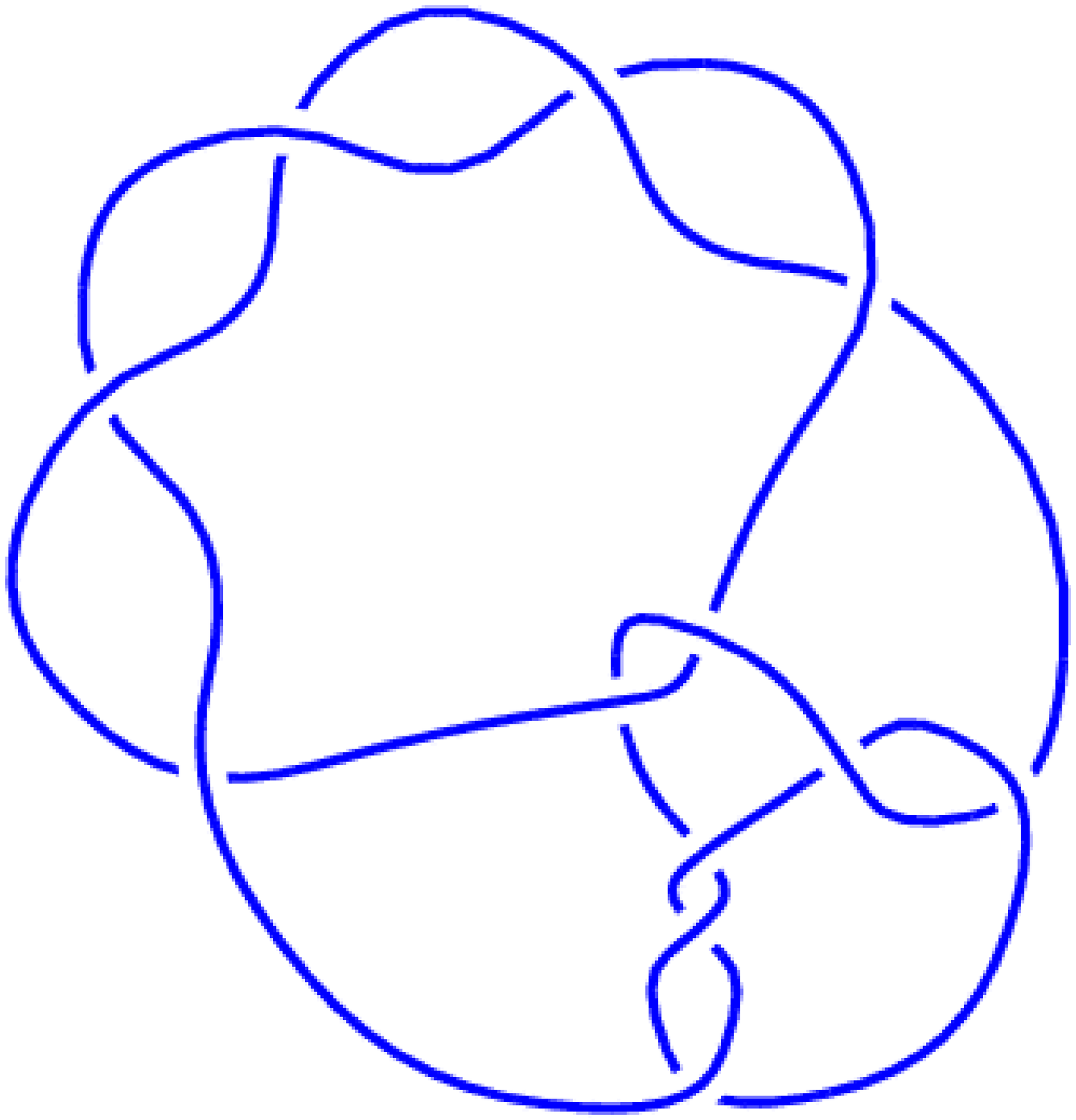}
\includegraphics[scale=0.14]{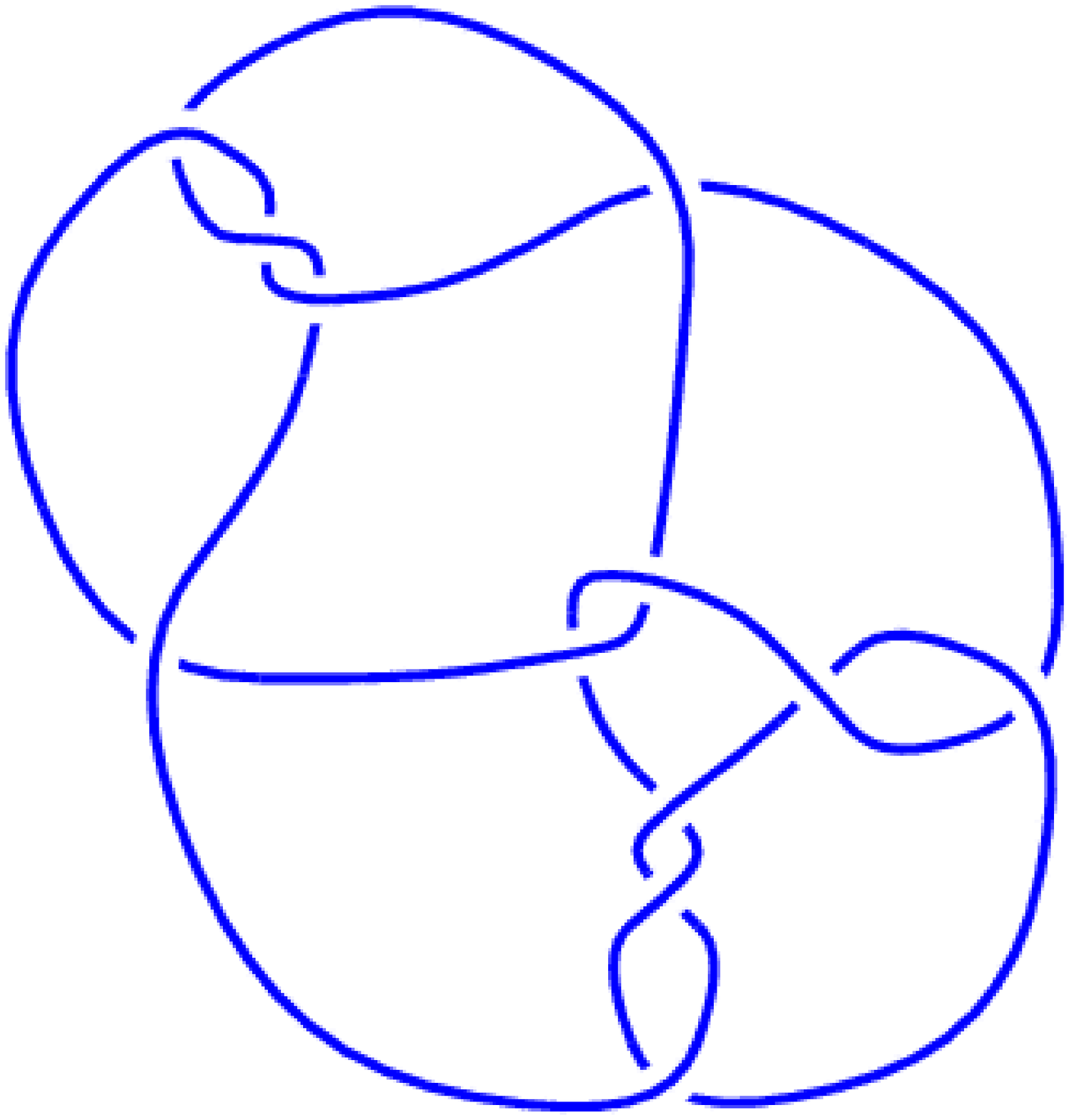}
\includegraphics[scale=0.14]{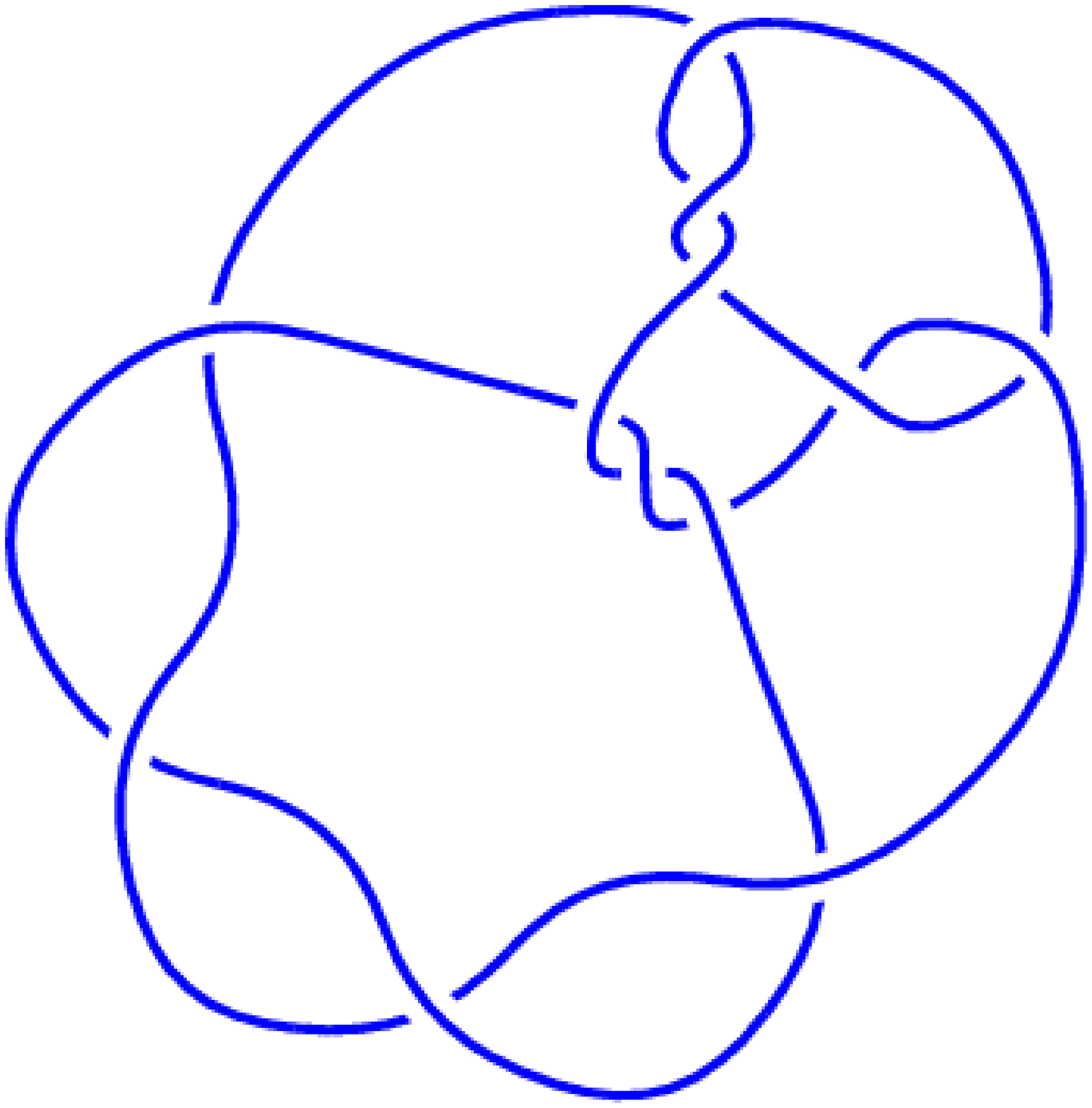}
\includegraphics[scale=0.14]{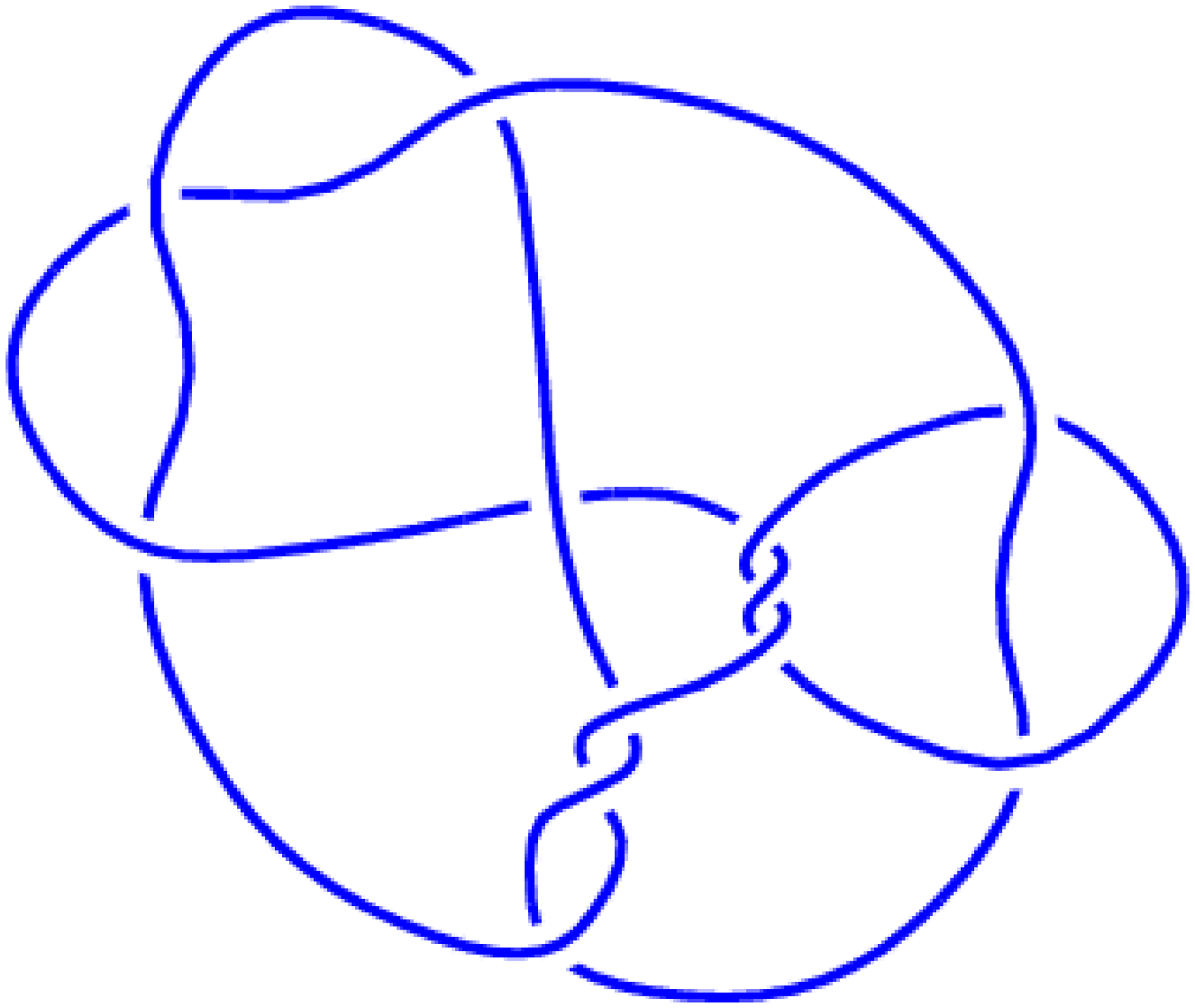}
\caption{The diagrams of the knots $10_{152}$, $12n437$, $12n558$, $12n679$, $12n680$,
$12n688$, and $12n689$, respectively.}
\label{figure7} \hspace{-1cm}
    \end{center}
\end{figure}
\end{ex}

\end{document}